\documentclass[a4paper]{article}

\usepackage{amssymb,amsmath,amsthm}
\usepackage{graphicx}
\usepackage{proof}
\usepackage{cite}
\usepackage[UKenglish]{babel}
\usepackage[all]{xy} 
\usepackage{hyperref}
\usepackage{times}

\hypersetup{
   colorlinks,%
   citecolor=blue,%
   filecolor=black,%
   linkcolor=red,%
   urlcolor=black
 }

\newcommand{\lr}[1]{\langle #1 \rangle}
\newcommand{\lra}{\leftrightarrow}

\newcommand{\bis}{\mathrel{\mathchoice%
{\raisebox{.3ex}{$\,
  \underline{\makebox[.7em]{$\leftrightarrow$}}\,$}}%
{\raisebox{.3ex}{$\,
  \underline{\makebox[.7em]{$\leftrightarrow$}}\,$}}%
{\raisebox{.2ex}{$\,
  \underline{\makebox[.5em]{\scriptsize$\leftrightarrow$}}\,$}}%
{\raisebox{.2ex}{$\,
  \underline{\makebox[.5em]{\scriptsize$\leftrightarrow$}}\,$}}}}

\newcommand{\TAUT}{\ensuremath{\texttt{TAUT}}}

\newcommand{\MP}{\ensuremath{\texttt{MP}}}
\newcommand{\REKw}{\ensuremath{\texttt{RE}\Delta}}
\newcommand{\EquiKw}{\ensuremath{{\Delta\texttt{Equ}}}}
\newcommand{\KwCon}{\ensuremath{\Delta\texttt{Con}}}

\newcommand{\KwDis}{\ensuremath{\Delta\texttt{Dis}}}
\newcommand{\KwTop}{\ensuremath{\Delta\texttt{Top}}}

\newcommand{\KwM}{\ensuremath{\Delta\texttt{M}}}
\newcommand{\KwC}{\ensuremath{\Delta\texttt{C}}}
\newcommand{\KwN}{\ensuremath{\Delta\texttt{N}}}
\newcommand{\KwD}{\ensuremath{\Delta\texttt{D}}}

\newcommand{\KwT}{\ensuremath{\Delta\texttt{T}}}

\newcommand{\SPLKw}{\ensuremath{\mathbf{K}^\Delta}}

\newcommand{\NCL}{\ensuremath{\mathbf{E}^\Delta}}

\newcommand{\ENCL}{\mathbf{E}^\Delta}
\newcommand{\CNCL}{\mathbf{M}^\Delta}
\newcommand{\MNCL}{\mathbf{C}^\Delta}
\newcommand{\CL}{\mathbf{CL}}

\newcommand{\BP}{\textbf{Prop}}

\newcommand{\M}{\ensuremath{\mathcal{M}}}
\newcommand{\F}{\ensuremath{\mathcal{F}}}

\renewcommand{\phi}{\varphi}

\newcommand{\red}[1]{\textcolor{red}{#1}}
\newcommand{\noteJF}[1]{~\red{(\textsf{Jie:} \texttt{#1})}}

\newcommand{\weg}[1]{}

\newtheorem{theorem}{Theorem}
\newtheorem{lemma}[theorem]{Lemma}
\newtheorem{definition}[theorem]{Definition}

\newtheorem{proposition}[theorem]{Proposition}
\newtheorem{example}[theorem]{Example}
\newtheorem{corollary}[theorem]{Corollary}

\begin{document}

\title{Neighborhood Contingency Logic:\\ A New Perspective\thanks{This research is supported by the youth project 17CZX053 of National Social Science Fundation of China. The author thanks Yanjing Wang for proposing the notion of quasi-filter structures and discussing on earlier versions of this manuscript. Thanks also go to two anonymous referees of NCML 2017 for their insightful comments and suggestions. An earlier version of the manuscript was presented on the conference of NCML 2017 at Zhejiang University in Oct. 2017.}}
\date{}
\author{Jie Fan\\
\small School of Philosophy, Beijing Normal University  \\
\small fanjie@bnu.edu.cn}
\maketitle

\begin{abstract}
In this paper, we propose a new neighborhood semantics for contingency logic, by introducing a simple property in standard neighborhood models. This simplifies the neighborhood semantics given in Fan and van Ditmarsch~\cite{FanvD:neighborhood}, but does not change the set of valid formulas. Under this perspective, among various notions of bisimulation and respective Hennessy-Milner Theorems, we show that $c$-bisimulation is equivalent to nbh-$\Delta$-bisimulation in the literature, which guides us to understand the essence of the latter notion. This perspective also provides various frame definability and axiomatization results.
\end{abstract}


\section{Introduction}

Under Kripke semantics, contingency logic ({\bf CL} for short) is non-normal, less expressive than standard modal logic ({\bf ML} for short), and the five basic frame properties (seriality, reflexivity, transitivity, symmetry, Eucludicity) cannot be defined in {\bf CL}. This makes the axiomatizations of {\bf CL} nontrivial: although there have been a mountain of work on the axiomatization problem since the 1960s~\cite{MR66,Humberstone95,DBLP:journals/ndjfl/Kuhn95,DBLP:journals/ndjfl/Zolin99,hoeketal:2004}, over $\mathcal{K}$, $\mathcal{D}$, $\mathcal{T}$, $4$, $5$ and any combinations thereof, no method in the cited work can uniformly handle all the five basic frame properties. This job has not been addressed until in~\cite{Fanetal:2015}, which also contains an axiomatization of {\bf CL} on $\mathcal{B}$ and its multi-modal version. This indicates that Kripke semantics may not be suitable for {\bf CL}.

Partly inspired by the above motivation (in particular, the non-normality of {\bf CL}), and partly by a weaker logical omniscience in Kripke semantics (namely, all theorems are known to be true or known to be false), a neighborhood semantics for {\bf CL} is proposed in~\cite{FanvD:neighborhood}, which interprets the non-contingency operator $\Delta$ in a way such that its philosophical intuition, viz. necessarily true or necessarily false, holds. However, under this (old) semantics, as shown in~\cite{FanvD:neighborhood}, {\bf CL} is still less expressive than {\bf ML} on various classes of neighborhood models, and many usual neighborhood frame properties are undefinable in {\bf CL}. Moreover, based on this semantics, \cite{Bakhtiarietal:2017} proposes a bisimulation (called `nbh-$\Delta$-bisimulation' there) to characterize {\bf CL} within {\bf ML} and within first-order logic ({\bf FOL} for short), but the essence of the bisimulation seems not quite clear.

In retrospect, no matter whether the semantics for {\bf CL} is Kripke-style or neigborhood-style in the sense of~\cite{FanvD:neighborhood}, there is an asymmetry between the syntax and models of {\bf CL}: on the one hand, the language is too weak, since it is less expressive than {\bf ML} over various model classes; on the other hand, the models are too strong, since its models are the same as those of {\bf ML}. This makes it hard to handle {\bf CL}.\footnote{Analogous problem occurs in the setting of knowing-value logic~\cite{wangetal:2013,WF:2014}.}

Inspired by \cite{GuWang:2016}, we simplify the neighborhood semantics for {\bf CL} in~\cite{FanvD:neighborhood}, and meanwhile keep the logic (valid formulas) the same by restricting models. This can weaken the too strong models so as to balance the syntax and models for {\bf CL}. Under this new perspective, we can gain a lot of things, for example, bisimulation notions and their corresponding Hennessy-Milner Theorems, and frame definability. Moreover, we show that one of bisimulation notions is equivalent to the notion of nbh-$\Delta$-bisimulation, which helps us understand the essence of nbh-$\Delta$-bisimulation. We also obtain some simple axiomatizations.


\section{Preliminaries}

\subsection{Language and old neighborhood semantics}

First, we introduce the language and the old neighborhood semantics of {\bf CL}. Fix a countable set $\textbf{Prop}$ of propositional variables. The language of {\bf CL}, denoted $\mathcal{L}_\Delta$, is an extension of propositional logic with a sole primitive modality $\Delta$, where $p\in\textbf{Prop}$.
$$\phi::=p\mid \neg\phi\mid (\phi\land\phi)\mid \Delta\phi$$
$\Delta\phi$ is read ``it is non-contingent that $\phi$''. $\nabla\phi$, read ``it is contingent that $\phi$'', abbreviates $\neg\Delta\phi$.

A neighborhood model for $\mathcal{L}_\Delta$ is defined as that for the language of {\bf ML}. That is, to say $\M=\lr{S,N,V}$ is a neighborhood model, if $S$ is a nonempty set of states, $N:S\to 2^{2^S}$ is a neighborhood function assigning each state in $S$ a set of neighborhoods, and $V:\textbf{Prop}\to 2^S$ is a valuation assigning each propositional variable in $\textbf{Prop}$ a set of states in which it holds. A neighborhood frame is a neighborhood model without any valuation.

There are a variety of neighborhood properties. The following list is taken from~\cite[Def.~3]{FanvD:neighborhood}.
\begin{definition}[Neighborhood properties]\label{def.properties}\

$(n)$: $N(s)$ \emph{contains the unit}, if $S\in N(s)$. 

$(r)$: $N(s)$ \emph{contains its core}, if $\bigcap N(s)\in N(s)$\weg{, where $\bigcap N(s)=\bigcap\{X\mid X\in N(s)\}$}. 

$(i)$: $N(s)$ \emph{is closed under intersections}, if $X,Y\in N(s)$ implies $X\cap Y\in N(s)$. 

$(s)$: $N(s)$ is \emph{supplemented}, or \emph{closed under supersets}, if $X\in N(s)$ and $X\subseteq Y\subseteq S$ implies $Y\in N(s)$. We also call this property `monotonicity'. 

$(c)$: $N(s)$ is\emph{ closed under complements}, if $X\in N(s)$ implies $S\backslash X\in N(s)$. 

$(d)$: $X\in N(s)$ implies $S\backslash X\notin N(s)$.

$(t)$:  $X\in N(s)$ implies $s\in X$.

$(b)$: $s\in X$ implies $\{u\in S\mid S\backslash X\notin N(u)\}\in N(s)$.

$(4)$:  $X\in N(s)$ implies $\{u\in S \mid X\in N(u)\}\in N(s)$.

$(5)$:  $X\notin N(s)$ implies $\{u\in S \mid X\notin N(u)\}\in N(s)$.
\end{definition}

Frame $\mathcal{F}=\lr{S,N}$ (and the corresponding model) possesses such a property P, if $N(s)$ has the property P for each $s\in S$, and we call the frame (resp. the model) P-frame (resp. P-model). 

Given a neighborhood model $\M=\lr{S,N,V}$ and $s\in S$, the old neighborhood semantics of $\mathcal{L}_\Delta$~\cite{FanvD:neighborhood} is defined as follows, where $\phi^{\M_\Vdash}=\{t\in S\mid \M,t\Vdash\phi \}$.
\[\begin{array}{lll}
\M,s\Vdash p&\text{iff}&s\in V(p)\\
\M,s\Vdash \neg\phi&\text{iff}&\M,s\nVdash\phi\\
\M,s\Vdash\phi\land\psi&\text{iff}&\M,s\Vdash\phi\text{ and }\M,s\Vdash\psi\\
\M,s\Vdash\Delta\phi & \text{iff}& \phi^{\M_\Vdash} \in N(s)\text{ or }(\neg\phi)^{\M_\Vdash}\in N(s)\\
\end{array}\]

\weg{where the semantics of the operator $\Delta$ is inspired by the neighborhood semantics for $\Box$ and the philosophical intuition of non-contingency. Recall that the neighborhood semantics for $\Box$ is defined as follows.
\[
\begin{array}{lll}
\M,s\Vdash\Box\phi&\text{iff}&\phi^{\M_\Vdash}\in N(s)\\
\end{array}
\]
Besides, non-contingency intuitively means necessarily true or necessarily false. }

\subsection{Existing results on old neighborhood semantics}

Under the above old neighborhood semantics, it is shown in~\cite[Props.2-7]{FanvD:neighborhood} that on the class of $(t)$-models or the class of $(c)$-models, $\mathcal{L}_\Delta$ is equally expressive as $\mathcal{L}_\Box$; however, on other class of models in Def.~\ref{def.properties}, $\mathcal{L}_\Delta$ is less expressive than $\mathcal{L}_\Box$; moreover, {\em none} of frame properties in the above list is definable in {\bf CL}.

\weg{Under this old neighborhood semantics, it is shown in~\cite[Props.2-6]{FanvD:neighborhood} that
\begin{proposition}
On the class of $(t)$-models or the class of $(c)$-models, $\mathcal{L}_\Delta$ is equally expressive as $\mathcal{L}_\Box$; however, on other class of models, $\mathcal{L}_\Delta$ is less expressive than $\mathcal{L}_\Box$.
\end{proposition}

And it is also shown in~\cite[Prop.7]{FanvD:neighborhood} that any frame property in the above list is undefinable in CL.
\begin{proposition}
The frame properties $(n)$, $(r)$, $(i)$, $(s)$, $(c)$, $(d)$, $(t)$, $(b)$, $(4)$, and $(5)$ are not definable in $\mathcal{L}_\Delta$.
\end{proposition}}

Based on the above semantics for {\bf CL}, a notion of bisimulation is proposed in~\cite{Bakhtiarietal:2017}, which is inspired by the definition of {\em precocongruences} in~\cite{hansenetal:2009} and the old neighbourhood semantics of $\Delta$.

\begin{definition}[nbh-$\Delta$-bisimulation]
Let $\M=\lr{S,N,V}$ and $\M'=\lr{S',N',V'}$ be neighborhood models. A nonempty relation $Z\subseteq S\times S'$ is a {\em nbh-$\Delta$-bisimulation} between $\M$ and $\M'$, if for all $(s,s')\in Z$,

\textbf{(Atoms)} $s\in V(p)$ iff $s'\in V'(p)$ for all $p\in \BP$;

\textbf{(Coherence)} if the pair $(U,U')$ is $Z$-coherent,\footnote{Let $R$ be a binary relation. We say $(U,U')$ is {\em $R$-coherent}, if for any $(x,y)\in R$, we have $x\in U$ iff $y\in U'$. We say $U$ is $R$-closed, if $(U,U)$ is $R$-coherent. It is obvious that $(\emptyset,\emptyset)$ is $R$-coherent for any $R$.} then
$$(U\in N(s)\text{ or }S\backslash U\in N(s))\text{ iff }(U'\in N'(s') \text{ or }S'\backslash U'\in N'(s')).$$
$(\M,s)$ and $(\M',s')$ is {\em nbh-$\Delta$-bisimilar}, notation $(\M,s)\sim_\Delta(\M',s')$, if there is a nbh-$\Delta$-bisimulation between $\M$ and $\M'$ containing $(s,s')$.\footnote{In fact, the notion of nbh-$\Delta$-bisimilarity is defined in a more complex way in~\cite{Bakhtiarietal:2017}. It is easy to show that our definition is equivalent to, but simpler than, that one.}
\end{definition}

Although it is inspired by both the definition of {\em precocongruences} in~\cite{hansenetal:2009} and the old neighbourhood semantics of $\Delta$, the essence of nbh-$\Delta$-bisimulation seems not so clear.

It is then proved that Hennessy-Milner Theorem holds for nbh-$\Delta$-bisimulation. For this, a notion of $\Delta$-saturated model is required.
\begin{definition}[$\Delta$-saturated model]\cite[Def.~11]{Bakhtiarietal:2017}\label{def.delta-satu-model}
Let $\M=\lr{S,N,V}$ be a neighborhood model. A set $X\subseteq S$ is $\Delta$-compact, if every set of $\mathcal{L}_\Delta$-formulas that is finitely satisfiable in $X$ is itself also satisfiable in $X$. $\M$ is said to be {\em $\Delta$-saturated}, if for all $s\in S$ and all $\equiv_{\mathcal{L}_\Delta}$-closed neighborhoods $X\in N(s)$, both $X$ and $S\backslash X$ are $\Delta$-compact.
\end{definition}

\begin{theorem}[Hennessy-Milner Theorem for nbh-$\Delta$-bisimulation]~\cite[Thm.1]{Bakhtiarietal:2017}\label{thm.hm-nbh-delta-bis}
On $\Delta$-saturated models $\M$ and $\M'$ and states $s$ in $\M$ and $s'$ in $\M'$, if $(\M,s)\equiv_{\mathcal{L}_\Delta}(\M',s')$, then $(\M,s)\sim_\Delta(\M',s').$
\end{theorem}

\weg{\begin{theorem}\label{thm.nbhbis-vanbenthem}\cite[Thm.2, Thm.3]{Bakhtiarietal:2017}
A modal formula is equivalent to an $\mathcal{L}_\Delta$-formula over the class of neighborhood models iff it is invariant under nbh-$\Delta$-bisimulation, and a first-order $\mathcal{L}_1$-formula is equivalent to an $\mathcal{L}_\Delta$-formula over the class of neighborhood models iff it is invariant under nbh-$\Delta$-bisimulation.
\end{theorem}}

\section{A new semantics for CL}
\weg{Throughout this paper, we let $\textbf{Prop}$ to be the set of propositional variables and $p\in\textbf{Prop}$. The language $\CL$ of contingency logic is defined recursively as follows:
$$\phi::=p\mid \neg\phi\mid (\phi\land\phi)\mid \Delta\phi$$

Let $\M=\lr{S,N,V}$ and let $\phi^\M=\{s\in S\mid \M,s\vDash\phi \}.$

\[\begin{array}{lll}
\M,s\Vdash\Delta\phi &\text{iff}& \phi^\M \in N(s)\\
\M,s\Vdash^o\Delta\phi & \text{iff}& \phi^\M \in N(s)\text{ or }(\neg\phi)^\M\in N(s)\\
\M,s\Vdash\Box\phi &\text{iff}& \phi^\M \in N_\Box(s)\\
\end{array}\]
Where the superscript $o$ stands for the Original neighborhood semantics proposed in~\cite{FanvD:neighborhood}.}


As mentioned above, there is an asymmetry between the syntax and neighborhood models of {\bf CL}, which makes it hard to tackle {\bf CL}. In this section, we propose a new neighborhood semantics for this logic. This semantics is simpler than the old one, but the two semantics are equivalent in that their logics (valid formulas) are the same.


The new neighborhood semantics $\Vvdash$ and the old one $\Vdash$ differ only in the case of non-contingency operator.
\[\begin{array}{lll}
\M,s\Vvdash\Delta\phi &\text{iff}& \phi^\M \in N(s),\\
\end{array}\]
where $\phi^\M=\{t\in\M\mid \M,t\Vvdash\phi\}$. To say two models with the same domain are {\em pointwise equivalent}, if every world in both models satisfies the same formulas.


We hope that although we change the semantics, the validities are still kept the same as the old one. So how to make it out? Recall that non-contingency means necessarily true or necessarily false, which implies that $\Delta p\lra\Delta\neg p$ should be valid. However, although the formula is indeed valid under the old neighborhood semantics, it is invalid under the new one\weg{, which also indicate that the philosophical intuition is not persisted}. In order to make this come about, we need make some restriction to the models. Look at a proposition first.

\begin{proposition}\label{prop.def-c}
Under the new semantics, $\Delta p\lra\Delta\neg p$ defines the property $(c)$.
\end{proposition}

\begin{proof}
Let $\F=\lr{S,N}$ be a neighborhood frame.

First, suppose $\F$ possesses $(c)$, we need to show $\F\Vvdash\Delta p\lra\Delta\neg p$. For this, assume any model $\M$ based on $\F$ and $s\in S$ such that $\M,s\Vvdash\Delta p$, thus $p^\M\in N(s)$. By $(c)$, $S\backslash p^\M\in N(s)$, i.e., $(\neg p)^\M\in N(s)$, which means exactly $\M,s\Vvdash\Delta\neg p$. Now assume $\M,s\Vvdash\Delta\neg p$, we have $(\neg p)^\M\in N(s)$, that is $S\backslash p^\M\in N(s)$. By $(c)$, $S\backslash(S\backslash p^\M)\in N(s)$, i.e. $p^\M\in N(s)$, and thus $\M,s\Vvdash\Delta p$. Hence $\M,s\Vvdash\Delta p\lra\Delta\neg p$, and therefore $\F\Vvdash\Delta p\lra\Delta\neg p$.

Conversely, suppose $\F$ does not possess $(c)$, we need to show $\F\not\Vvdash\Delta p\lra\Delta\neg p$. By supposition, there exists $X$ such that $X\in N(s)$ but $S\backslash X\notin N(s)$. Define a valuation $V$ on $\F$ as $V(p)=X$. We have now $p^\M=V(p)\in N(s)$, thus $\M,s\Vvdash\Delta p$. On the other side, $V(\neg p)=S\backslash X\notin N(s)$, thus $\M,s\not\Vvdash \Delta\neg p$. Hence $\M,s\not\Vvdash\Delta p\to\Delta\neg p$, and therefore $\F\not\Vvdash\Delta p\lra\Delta\neg p$.
\end{proof}

This means that in order to guarantee the validity $\Delta p\lra\Delta\neg p$ under new semantics, we (only) need to `force' the model to have the property $(c)$. Thus from now on, we assume $(c)$ to be the {\em minimal} condition of a neighborhood model, and call this type of model `$c$-models'.\weg{\footnote{This assumption is reasonable, because the formula $\Delta\phi\lra\Delta\neg\phi$ will enable the canonical model to have the property $(c)$, as in~\cite{FanvD:neighborhood}.}}

\begin{definition}[$c$-structures]
A model is a {\em $c$-model}, if it has the property $(c)$; intuitively, if a proposition is non-contingent at a state in the domain, so is its negation. A frame is a {\em $c$-frame}, if the models based on it are $c$-models.
\end{definition}


The following proposition states that on $c$-models, the new neighborhood semantics and the old one coincide with each other in terms of $\mathcal{L}_\Delta$ satisfiability.
\begin{proposition}\label{prop.satis-same}
Let $\M=\lr{S,N,V}$ be a $c$-model. Then for all $\phi\in\mathcal{L}_\Delta$, for all $s\in S$, we have $\M,s\Vvdash\phi\iff\M,s\Vdash\phi$, i.e., $\phi^\M=\phi^{\M_\Vdash}$.
\end{proposition}

\begin{proof}
By induction on $\phi\in\mathcal{L}_\Delta$. The only nontrivial case is $\Delta\phi$.

First, suppose $\M,s\Vvdash\Delta\phi$, then $\phi^\M\in N(s)$. By induction hypothesis, $\phi^{\M_\Vdash}\in N(s)$. Of course, $\phi^{\M_\Vdash}\in N(s)$ or $(\neg\phi)^{\M_\Vdash}\in N(s)$. This entails that $\M,s\Vdash\Delta\phi$.

Conversely, assume $\M,s\Vdash\Delta\phi$, then $\phi^{\M_\Vdash}\in N(s)$ or $(\neg\phi)^{\M_\Vdash}\in N(s)$. Since $\M$ is a $c$-model, we can obtain $\phi^{\M_\Vdash}\in N(s)$. By induction hypothesis, $\phi^\M\in N(s)$.  Therefore, $\M,s\Vvdash\Delta\phi$.
\end{proof}

\begin{definition}[$c$-variation] Let $\M=\lr{S,N,V}$ be a neighborhood model. We say $c(\M)$ is a {\em $c$-variation} of $\M$, if $c(\M)=\lr{S,cN,V}$, where for all $s\in S$, $cN(s)=\{X\subseteq S:X\in N(s) \text{ or }S\backslash X\in N(s)\}$.
\end{definition}

The definition of $cN$ is very natural, in that just as ``$X\in N(s) \text{ or }S\backslash X\in N(s)$'' corresponds to the old semantics of $\Delta$, $X\in cN(s)$ corresponds to the new semantics of $\Delta$. It is easy to see that every neighborhood model has a sole $c$-variation, that every such $c$-variation is a c-model, and moreover, for any neighborhood model $\M$, if $\M$ is already a $c$-model, then $c(\M)=\M$.

\begin{proposition}\label{prop.neigh-c}
Let $\M$ be a neighborhood model. Then for all $\phi\in\mathcal{L}_\Delta$, for all $s\in\M$, we have $\M,s\Vdash\phi\iff c(\M),s\Vvdash\phi$, i.e., $\phi^{\M_\Vdash}=\phi^{c(\M)}$.
\end{proposition}

\begin{proof}
The proof is by induction on $\phi$, where the only nontrivial case is $\Delta\phi$. We have
\[
\begin{array}{lll}
\M,s\Vdash\Delta\phi&\iff&\phi^{\M_\Vdash}\in N(s)\text{ or }S\backslash(\phi^{\M_\Vdash})\in N(s)\\
&\stackrel{\text{IH}}\iff&\phi^{c(\M)}\in N(s)\text{ or }S\backslash(\phi^{c(\M)})\in N(s)\\
&\stackrel{\text{Def.}cN}\iff&\phi^{c(\M)}\in cN(s)\\
&\iff&c(\M),s\Vvdash\Delta\phi\\
\end{array}\]
\weg{Conversely, let $\mathcal{O}=\lr{S,N,V}$ be a $c$-model and $s\in S$. Define a model $\mathcal{O}'=\lr{S,N',V}$ such that $$N'=N.$$
We now proceed with induction on $\phi$, to show that $\mathcal{O}',s\Vdash^o\phi\iff\mathcal{O},s\Vdash\phi$ for all $\phi\in\CL$. We only need to check the case $\Delta\phi$.
\[
\begin{array}{lll}
\mathcal{O}',s\Vdash^o\Delta\phi&\iff& \phi^{\mathcal{O}'}\in N'(s)\text{ or }S\backslash(\phi^{\mathcal{O}'})\in N'(s)\\
&\stackrel{\text{IH,Def.~}N'}\iff & \phi^{\mathcal{O}}\in N(s)\text{ or }S\backslash(\phi^{\mathcal{O}})\in N(s)\\
&\stackrel{c}\iff&\phi^{\mathcal{O}}\in N(s)\\
&\iff&\mathcal{O},s\Vdash\Delta\phi\\
\end{array}
\]}
\end{proof}

Let $\Gamma\Vvdash_c\phi$ denote that $\Gamma$ entails $\phi$ over the class of all $c$-models, i.e., for each $c$-model $\M$ and each $s\in \M$, if $\M,s\Vvdash\psi$ for every $\psi\in\Gamma$, then $\M,s\Vvdash\phi$. With Props.~\ref{prop.satis-same} and~\ref{prop.neigh-c} in hand, we obtain immediately that
\begin{corollary}\label{coro.c-valid-reduce}
For all $\Gamma\cup\{\phi\}\subseteq\mathcal{L}_\Delta$, $\Gamma\Vvdash_c\phi\iff\Gamma\Vdash\phi.$ Therefore, for all $\phi\in\mathcal{L}_\Delta$, $\Vvdash_c\phi\iff\Vdash\phi.$
\end{corollary}

\weg{Therefore, if we would like to show the completeness proof for the old neighborhood semantics, we need only show the completeness proof for the new one.

\begin{corollary}
Let $P$ be any neighborhood model property. For all $\phi\in\CL$, we have $$\Vdash_{Pc}\phi\iff\Vdash^o_{Pc}\phi.$$
\end{corollary}

\begin{proof}
We only need to prove $$\Vdash^n_{c}\phi\iff\Vdash_{c}\phi.$$
\end{proof}

Unfortunately, we do not have $\Vdash^n_{Pc}\phi\iff\Vdash_{P}\phi.$ For example, $\Vdash^n_{sc}\Delta(\phi\land\psi)\to\Delta\phi\land\Delta\psi$, but we do not have $\Vdash_s\Delta(\phi\land\psi)\to\Delta\phi\land\Delta\psi$.}

In this way, we strengthened the expressive power of {\bf CL}, since it is now equally expressive as {\bf ML}, and kept the logic (valid formulas) the same as the old neighborhood semantics. The noncontingency operator $\Delta$ can now be seen as a package of $\Box$ and $\Delta$ in the old neighborhood semantics; under the new neighborhood semantics, on the one hand, it is interpreted just as $\Box$; on the other hand, it retains the validity $\Delta p\lra \Delta\neg p$.

\section{$c$-Bisimulation}\label{sec.c-bis}

Recall that the essence of the notion of nbh-$\Delta$-bisimulation proposed in~\cite{Bakhtiarietal:2017} is not so clear. In this section, we introduce a notion of $c$-bisimulation, and show that this notion is equivalent to nbh-$\Delta$-bisimulation. The $c$-bisimulation is inspired by both Prop.~\ref{prop.def-c} and the definition of {\em precocongruences} in~\cite[Prop.~3.16]{hansenetal:2009}. Intuitively, the notion is obtained by just adding the property $(c)$ into the notion of precocongruences.


\weg{First we briefly review the nbh-$\Delta$-bisimulation introduced in~\cite[Def.9]{Bakhtiarietal:2017}, which was inspired by the definition of {\em precocongruences} in~\cite{hansenetal:2009} and the neighbourhood semantics of $\Delta$.

\begin{definition}[nbh-$\Delta$-bisimulation]
Let $\M=\lr{S,N,V}$ and $\M'=\lr{S',N',V'}$ be neighborhood models. A nonempty relation $Z\subseteq S\times S'$ is a {\em nbh-$\Delta$-bisimulation}, if for all $(s,s')\in Z$,

\textbf{(Atoms)} $s\in V(p)$ iff $s'\in V'(p)$ for all $p\in \BP$;

\textbf{(Coherence)} if the pair $(U,U')$ is $Z$-coherent, then
$$(U\in N(s)\text{ or }S\backslash U\in N(s))\text{ iff }(U'\in N'(s') \text{ or }S'\backslash U'\in N'(s')).$$
\end{definition}}

\begin{definition}[c-bisimulation]
Let $\M=\lr{S,N,V}$ and $\M'=\lr{S',N',V'}$ be $c$-models. A nonempty relation $Z\subseteq S\times S'$ is a {\em c-bisimulation} between $\M$ and $\M'$, if for all $(s,s')\in Z$,
\begin{enumerate}
\item[(i)] $s\in V(p)$ iff $s'\in V'(p)$ for all $p\in \BP$;
\item[(ii)] if the pair $(U,U')$ is $Z$-coherent, then
$U\in N(s)\text{ iff }U'\in N'(s').$
\end{enumerate}
We say $(\M,s)$ and $(\M',s')$ are {\em c-bisimilar}, written $(\M,s)\bis_c(\M',s')$, if there is a c-bisimulation $Z$ between $\M$ and $\M'$ such that $(s,s')\in Z$.
\end{definition}

Note that both $c$-bisimulation and $c$-bisimilarity are defined between $c$-models, rather than between any neighborhood models. $\mathcal{L}_\Delta$ formulas are invariant under $c$-bisimilarity.
\begin{proposition}\label{prop.invariance-c-bis}
Let $\M$ and $\M'$ be $c$-models, $s\in \M$ and $s'\in \M'$. If $(\M,s)\bis_c(\M',s')$, then for all $\phi\in\mathcal{L}_\Delta$, $\M,s\Vvdash \phi\iff\M',s'\Vvdash\phi.$
\end{proposition}

\begin{proof}
Let $\M=\lr{S,N,V}$ and $\M'=\lr{S',N',V'}$ be both $c$-models. By induction on $\phi\in\mathcal{L}_\Delta$. The nontrivial case is $\Delta\phi$.
\[
\begin{array}{ll}
&\M,s\Vvdash\Delta\phi\\
\iff&\phi^\M\in N(s)\\
\stackrel{(\ast)}\iff&\phi^{\M'}\in N'(s')\\
\iff&\M',s'\Vvdash\Delta\phi.\\
\end{array}
\]
$(\ast)$ follows from the fact that $(\phi^\M,\phi^{\M'})$ is $\bis_c$-coherent plus the condition $(ii)$ of $c$-bisimulation. To see why $(\phi^\M,\phi^{\M'})$ is $\bis_c$-coherent, the proof goes as follows: if for any $(x,x')\in\bis_c$, i.e., $(\M,x)\bis_c(\M',x')$, then by induction hypothesis, $\M,x\Vvdash\phi$ iff $\M',x'\Vvdash\phi$, i.e., $x\in \phi^\M$ iff $x'\in\phi^{\M'}$.
\end{proof}


Now we are ready to show the Hennessy-Milner Theorem for $c$-bisimulation. Since $c$-bisimulation is defined between $c$-models, we need also to add the property $c$ into the notion of $\Delta$-saturated models in Def.~\ref{def.delta-satu-model}.
\begin{definition}[$\Delta$-saturated $c$-model]
Let $\M=\lr{S,N,V}$ be a $c$-model. A set $X\subseteq S$ is $\Delta$-compact, if every set of $\mathcal{L}_\Delta$-formulas that is finitely satisfiable in $X$ is itself also satisfiable in $X$. $\M$ is said to be {\em $\Delta$-saturated}, if for all $s\in S$ and all $\equiv_{\mathcal{L}_\Delta}$-closed neighborhood $X\in N(s)$, \weg{both }$X$\weg{ and $S\backslash X$ are} is $\Delta$-compact.\footnote{Note that we do not distinguish $\equiv_{\mathcal{L}_\Delta}$ here from that in Def.~\ref{def.delta-satu-model} despite different neighborhood semantics. This is because as we show in Prop.~\ref{prop.satis-same}, on $c$-models the two neighborhood semantics are the same in terms of $\mathcal{L}_\Delta$ satisfiability. Thus it does not matter which semantics is involved in the current context.\weg{We use $\equiv_{\mathcal{L}_\Delta}$ to mean $\mathcal{L}_\Delta$-equivalent.}}
\end{definition}

In the above definition of $\Delta$-saturated $c$-model, we write ``$X$ is $\Delta$-compact'', rather than ``both $X$ and $S\backslash X$ are $\Delta$-compact'', since under the condition that $X\in N(s)$ and the property $(c)$, these two statements are equivalent. Thus each $\Delta$-saturated $c$-model must be a $\Delta$-saturated model.

We will demonstrate that on $\Delta$-saturated $c$-models, $\mathcal{L}_\Delta$-equivalence implies $c$-bisimilarity, for which we prove that the notion of c-bisimulation is equivalent to that of nbh-$\Delta$-bisimulation, in the sense that every nbh-$\Delta$-bisimulation (between neighborhood models) is a c-bisimulation (between $c$-models), and vice versa. By doing so, we can see clearly the essence of nbh-$\Delta$-bisimulation, i.e. precocongruences with property $(c)$. 

\begin{proposition}\label{prop.nbh-c}
Let $\M=\lr{S,N,V}$ and $\M'=\lr{S',N',V'}$ be neighborhood models.
If $Z$ is a nbh-$\Delta$-bisimulation between $\M$ and $\M'$, then $Z$ is a c-bisimulation between $c(\M)$ and $c(\M')$.\weg{, where $c(\M)$ and $c(\M')$ are c-variations of $\M$ and $\M'$, respectively. Therefore, for all $s\in\M$ and all $s'\in\M'$, if $(\M,s)\sim_\Delta(\M',s')$, then $(c(\M),s)\bis_c(c(\M'),s')$.}
\end{proposition}

\begin{proof}
Suppose that $Z$ is a nbh-$\Delta$-bisimulation between $\M$ and $\M'$, to show $Z$ is a c-bisimulation between $c(\M)$ and $c(\M')$.

First, one can easily verify that $c(\M)$ and $c(\M')$ are both $c$-models.

Second, assume that $(s,s')\in Z$. Since $\M$ and $c(\M)$ have the same domain and valuation, item (i) can be obtained from the supposition and \textbf{(Atoms)}. For item (ii), let $(U,U')$ be $Z$-coherent. We need to show that $U\in cN(s)$ iff $U'\in cN'(s')$. For this, we have the following line of argumentation: $U\in cN(s)$ iff (by definition of $cN$) ($U\in N(s)$ or $S\backslash U\in N(s)$) iff (by \textbf{(Coherence)}) iff ($U'\in N'(s')$ or $S'\backslash U'\in N'(s')$) iff (by definition of $cN'$) $U'\in cN'(s')$.
\end{proof}


\begin{proposition}\label{prop.c-nbh}
Let $\M=\lr{S,N,V}$ and $\M'=\lr{S',N',V'}$ be $c$-models. If $Z$ is a c-bisimulation between $\M$ and $\M'$, then $Z$ is a nbh-$\Delta$-bisimulation between $\M$ and $\M'$. 
\end{proposition}

\begin{proof}
Suppose that $Z$ is a c-bisimulation between $c$-models $\M$ and $\M'$, to show $Z$ is a nbh-$\Delta$-bisimulation between $\M$ and $\M'$. Assume that $(s,s')\in Z$, we only need to show \textbf{(Atoms)} and \textbf{(Coherence)} holds. \textbf{(Atoms)} is clear from (i).

For \textbf{(Coherence)}, let the pair $(U,U')$ is $Z$-coherent. Then by (ii), $U\in N(s)\text{ iff }U'\in N'(s')$. We also have that $(S\backslash U,S'\backslash U')$ is $Z$-coherent. Using (ii) again, we infer that $S\backslash U\in N(s)$ iff $S'\backslash U'\in N'(s')$. Therefore, ($U\in N(s)$ or $S\backslash U\in N(s)$) iff ($U'\in N'(s')$ or $S'\backslash U'\in N'(s')$), as desired.
\end{proof}

Since every $c$-variation of a $c$-model is just the model itself, by Props.~\ref{prop.nbh-c} and~\ref{prop.c-nbh}, we obtain immediately that
\begin{corollary}\label{coro.c-bis-equi-nbh-bis}
Let $\M=(S,N,V)$ and $\M'=\lr{S',N',V'}$ be both $c$-models. Then $Z$ is a $c$-bisimulation between $\M$ and $\M'$ iff $Z$ is an nbh-$\Delta$-bisimulation between $\M$ and $\M'$.
\end{corollary}

\begin{theorem}[Hennessy-Milner Theorem for $c$-bisimulation]
Let $\M$ and $\M'$ be $\Delta$-saturated $c$-models, and $s\in\M$, $s'\in\M'$. If for all $\phi\in\mathcal{L}_\Delta$, $\M,s\Vvdash \phi\iff\M',s'\Vvdash\phi$, then $(\M,s)\bis_c(\M',s')$.
\end{theorem}

\begin{proof}
Suppose $\M$ and $\M'$ are $\Delta$-saturated $c$-models such that for all $\phi\in\mathcal{L}_\Delta$, $\M,s\Vvdash \phi\iff\M',s'\Vvdash\phi$. By Prop.~\ref{prop.satis-same}, we have that for all $\phi\in\mathcal{L}_\Delta$, $\M,s\Vdash \phi\iff\M',s'\Vdash\phi$. Since each $\Delta$-saturated $c$-model is a $\Delta$-saturated model, by Hennessy-Milner Theorem of nbh-$\Delta$-bisimulation (Thm.~\ref{thm.hm-nbh-delta-bis}), we obtain $(\M,s)\sim_\Delta(\M',s')$. Then by Coro.~\ref{coro.c-bis-equi-nbh-bis}, we conclude that $(\M,s)\bis_c(\M',s')$.
\end{proof}

\section{Monotonic $c$-bisimulation}

This section proposes a notion of bisimulation for {\bf CL} over monotonic, $c$-models. This notion can be obtained via two ways: one is to add the property of monotonicity $(s)$ into $c$-bisimulation, the other is to add the property $(c)$ into monotonic bisimulation (for {\bf ML}).\footnote{For the notion of monotonic bisimulation, refer to~\cite[Def.~4.10]{hansen2003monotonic}.} For the sake of reference, we call the notion obtained by the first way `monotonic $c$-bisimulation', and that obtained by the second way `$c$-monotonic bisimulation'. We will show that the two notions are indeed the same.


\begin{definition}[Monotonic $c$-bisimulation]
Let $\M=\lr{S,N,V}$ and $\M'=\lr{S',N',V'}$ be both monotonic, $c$-models. A nonempty binary relation $Z$ is a {\em monotonic $c$-bisimulation} between $\M$ and $\M'$, if $sZs'$ implies the following:

(i) $s$ and $s'$ satisfy the same propositional variables;

(ii) If $(U,U')$ is $Z$-coherent, then $U\in N(s)\text{ iff }U'\in N'(s').$

$(\M,s)$ and $(\M',s')$ is said to be {\em monotonic $c$-bisimilar}, written $(\M,s)\bis_{sc}(\M',s')$, if there is a monotonic $c$-bisimulation between $\M$ and $\M'$ such that $sZs'$.
\end{definition}


\begin{definition}[$c$-monotonic bisimulation]
Let $\M=\lr{S,N,V}$ and $\M'=\lr{S',N',V'}$ be both monotonic, $c$-models. A nonempty binary relation $Z$ is a {\em $c$-monotonic bisimulation} between $\M$ and $\M'$, if $sZs'$ implies the following:

(Prop) $s$ and $s'$ satisfy the same propositional variables;

(c-m-Zig) if $X\in N(s)$, then there exists $X'\in N'(s')$ such that for all $x'\in X'$, there is an $x\in X$ such that $xZx'$;

(c-m-Zag) if $X'\in N'(s')$, then there exists $X\in N(s)$ such that for all $x\in X$, there is an $x'\in X'$ such that $xZx'$.

$(\M,s)$ and $(\M',s')$ is said to be {\em $c$-monotonic bisimilar}, written $(\M,s)\bis_{cs}(\M',s')$, if there is a $c$-monotonic bisimulation between $\M$ and $\M'$ such that $sZs'$.
\end{definition}

Note that both monotonic $c$-bisimulation and $c$-monotonic bisimulation are defined between monotonic, $c$-models.

\begin{proposition}\label{prop.cs-sc}
Every $c$-monotonic bisimulation is a monotonic $c$-bisimulation.
\end{proposition}

\begin{proof}
Suppose that $Z$ is a $c$-monotonic bisimulation between $\M$ and $\M'$, both of which are monotonic, $c$-models, to show that $Z$ is also a monotonic $c$-bisimulation between $\M$ and $\M'$. For this, assume that $sZs'$, it suffices to show the condition (ii).

Assume that $(U,U')$ is $Z$-coherent. If $U\in N(s)$, by (c-m-Zig), there exists $X'\in N'(s')$ such that for all $x'\in X'$, there is a $x\in U$ such that $xZx'$. By assumption and $x\in U$ and $xZx'$, we have $x'\in U'$, thus $X'\subseteq U'$. Then by $(s)$ and $X'\in N'(s')$, we conclude that $U'\in N'(s')$. The converse is similar, but by using (c-m-Zag) instead.
\end{proof}


\begin{proposition}
Every monotonic $c$-bisimulation is a $c$-monotonic bisimulation.
\end{proposition}

\begin{proof}
Suppose that $Z$ is a monotonic $c$-bisimulation between $\M$ and $\M'$, both of which are monotonic, $c$-models, to show that $Z$ is also a $c$-monotonic bisimulation between $\M$ and $\M'$. For this, given that $sZs'$, we need to show the condition (c-m-Zig) and (c-m-Zag). We show (c-m-Zig) only, since (c-m-Zag) is similar.

Assume that $X\in N(s)$, define $X'=\{x'\mid xZx'\text{~for some~}x\in X\}.$ 
It suffices to show that $X'\in N'(s')$.
~The proof is as follows: by assumption and monotonicity of $\M$, we have $S\in N(s)$, then by $(c)$, $\emptyset \in N(s)$. Since $(\emptyset,\emptyset)$ is $Z$-coherent, by (ii), we infer $\emptyset\in N'(s')$. From this and monotonicity of $\M'$, it follows that $X'\in N'(s')$, as desired.
\end{proof}

As a corollary, the aforementioned two ways enable us to get the same bisimulation notion.
\begin{corollary}
The notion of monotonic $c$-bisimulation is equal to the notion of $c$-monotonic bisimulation.
\end{corollary}

So we can choose either of the two bisimulation notions to refer to the notion of bisimulation of {\bf CL} over monotonic, $c$-models. In the sequel, we choose the simpler one, that is, monotonic $c$-bisimulation. One may easily see that this notion is stronger than monotonic bisimulation (for {\bf ML}). 

Similar to the case for $c$-bisimulation in Sec.~\ref{sec.c-bis}, we can show that

\begin{proposition}\label{prop.invariance-sc-bis}
Let $\M$ and $\M'$ be monotonic, $c$-models, $s\in \M$ and $s'\in \M'$. If $(\M,s)\bis_{sc}(\M',s')$, then for all $\phi\in\mathcal{L}_\Delta$, $\M,s\Vvdash \phi\iff\M',s'\Vvdash\phi.$
\end{proposition}

\weg{\begin{proposition}\label{prop.invariance-sc-bis}
Let $\M$ and $\M'$ be monotonic, $c$-models, $s\in \M$ and $s'\in \M'$. Then for every $s\in \M$ and $s'\in S'$, if $(\M,s)\bis_{cs}(\M',s')$, then for all $\phi\in\mathcal{L}_\Delta$, we have $$\M,s\Vvdash \phi\iff\M',s'\Vvdash\phi.$$
\end{proposition}

\begin{proof}
Suppose that $\M$ and $\M'$ are both monotonic, $c$-models, and $(\M,s)\bis_{cs}(\M',s')$, then there is a $c$-monotonic bisimulation $Z$ between $\M$ and $\M'$ such that $sZs'$. The proof continues with induction on $\phi\in\mathcal{L}_\Delta$. We need only consider the nontrivial case $\Delta\phi$.

Assume that $\M,s\Vvdash\Delta\phi$, then $\phi^\M\in N(s)$. By (c-m-Zig), there exists $X'\in N'(s')$ such that for all $x'\in X'$, there is an $x\in \phi^\M$ such that $xZx'$. By induction hypothesis, we obtain $x'\in \phi^{\M'}$, and thus $X'\subseteq \phi^{\M'}$. Using the monotonicity $(s)$, we infer $\phi^{\M'}\in N'(s')$, and therefore $\M',s'\Vvdash\Delta\phi$. The converse is similar, by using (c-m-Zag) instead.
\end{proof}}

\weg{\begin{proposition}\label{prop.invariance-sdelta-bis}
Let $\M$ and $\M'$ be monotonic models. Then for every $s\in \M$ and $s'\in S'$, if $(\M,s)\sim_{s\Delta}(\M',s')$, then for all $\phi\in\mathcal{L}_\Delta$, we have $$\M,s\Vdash \phi\iff\M',s'\Vdash\phi.$$
\end{proposition}}

\begin{theorem}[Hennessy-Milner Theorem for monotonic $c$-bisimulation]
Let $\M$ and $\M'$ be monotonic, $\Delta$-saturated $c$-models, $s\in \M$ and $s'\in \M'$. If for all $\phi\in\mathcal{L}_\Delta$, $\M,s\Vvdash \phi\iff\M',s'\Vvdash\phi$, then $(\M,s)\bis_{sc}(\M',s')$.
\end{theorem}

Similarly, we can define regular $c$-bisimulation, which is obtained by adding the property $(i)$ into monotonic $c$-bisimulation, and show the corresponding Hennessy-Milner Theorem. We omit the details due to space limitation.

\section{Quasi-filter structures}\label{Sec.qf-structures}

We define a class of structures, called `quasi-filter structures'.\footnote{Note that our notion of quasi-filter is different from that in~\cite[p.~215]{Chellas1980}, where quasi-filter is defined as $(s)+(i)$. For example, the latter notion is not necessarily closed under complements.} 

\begin{definition}[Quasi-filter frames and models]\label{def.ws}
A neighborhood frame $\mathcal{F}=\lr{S,N}$ is a {\em quasi-filter frame}, if for all $s\in S$, $N(s)$ possesses the properties $(n)$, $(i)$, $(c)$, and $(ws)$, where $(ws)$ means being closed under supersets or co-supersets: for all $X,Y,Z\subseteq S$, $X\in N(s)$ implies $X\cup Y\in N(s)$ or $(S\backslash X)\cup Z\in N(s)$.

We say a neighborhood model is a {\em quasi-filter model}, if its underlying frame is a quasi-filter frame.
\end{definition}

The main result of this section is the following: for {\bf CL}, every Kripke model has a pointwise equivalent quasi-filter model, but {\em not} vice versa. 



\begin{definition}[qf-variation]\label{def.qf-variation}
Let $\M=\lr{S,R,V}$ is a Kripke model. $qf(\M)$ is said to be a {\em qf-variation} of $\M$, if $qf(\M)=\lr{S,qfN,V}$, where for any $s\in S$, $qfN(s)=\{X\subseteq S: \text{ for any }t,u\in S,\text{ if }sRt\text{ and }sRu,\text{ then }(t\in X\text{ iff }u\in X)\}$.
\end{definition}

The definition of $qfN$ is also quite natural, since just as ``for any $t,u\in S$, if $sRt$ and $sRu$, then ($t\in X$ iff $u\in X$)'' corresponds to the Kripke semantics of $\Delta$, $X\in qfN(s)$ corresponds to the new neighborhood semantics of the operator, as will be seen more clearly in Prop.~\ref{prop.rel-equiv-qf}. Note that the definition of $qfN$ can be simplified as follows: $$qfN(s)=\{X\subseteq S:R(s)\subseteq X\text{~or~}R(s)\subseteq S\backslash X\}.$$
It is easy to see that every Kripke model has a (sole) qf-variation. We will demonstrate that, every such qf-variation is a quasi-filter model.

The following proposition states that every Kripke model and its qf-variation are pointwise equivalent.
\begin{proposition}\label{prop.rel-equiv-qf}
Let $\M=\lr{S,R,V}$ be a Kripke model. Then for all $\phi\in\mathcal{L}_\Delta$, for all $s\in S$,  we have
$\M,s\vDash\phi\iff qf(\M),s\Vvdash\phi,$ i.e., $\phi^{\M_\vDash}=\phi^{qf(\M)}$, where $\phi^{\M_\vDash}=\{t\in S\mid \M,t\vDash\phi\}$.
\end{proposition}

\begin{proof}
By induction on $\phi$. The nontrivial case is $\Delta\phi$.
\[
\begin{array}{lll}
\M,s\vDash\Delta\phi&\iff&\text{for all }t,u\in S, \text{ if }sRt\text{ and }sRu,\text{ then }\\
&&(t\in\phi^{\M_\vDash}\iff u\in\phi^{\M_\vDash})\\
&\stackrel{\text{IH}}\iff &\text{for all }t,u\in S, \text{ if }sRt\text{ and }sRu,\text{ then }\\
&&(t\in\phi^{qf(\M)}\iff u\in\phi^{qf(\M)})\\
&\stackrel{\text{Def.~}qfN}\iff&\phi^{qf(\M)}\in qfN(s)\\
&\iff&qf(\M),s\Vvdash\Delta\phi.\\
\end{array}
\]
\end{proof}

\begin{proposition}\label{prop.qf-model}
Let $\M$ be a Kripke model. Then \weg{its qf-variation }$qf(\M)$ is a quasi-filter model.
\end{proposition}

\begin{proof}
Let $\M=\lr{S,R,V}$. For any $s\in S$, we show that $qf(\M)$ has those four properties of quasi-filter models.

$(n)$: it is clear that $S\in qfN(s)$.

$(i)$: assume that $X,Y\in qfN(s)$, we show $X\cap Y\in qfN(s)$. By assumption, for all $s,t\in S$, if $sRt$ and $sRu$, then $t\in X$ iff $u\in X$, and for all $s,t\in S$, if $sRt$ and $sRu$, then $t\in Y$ iff $u\in Y$. Therefore, for all $t,u\in S$, if $sRt$ and $sRu$, we have that $t\in X\cap Y$ iff $u\in X \cap Y$. This entails $X\cap Y\in qfN(s)$.

$(c)$: assume that $X\in qfN(s)$, to show $S\backslash X\in qfN(s)$. By assumption, for all $s,t\in S$, if $sRt$ and $sRu$, then $t\in X$ iff $u\in X$. Thus for all $s,t\in S$, if $sRt$ and $sRu$, then $t\in S\backslash X$ iff $u\in S\backslash X$, i.e., $S\backslash X\in qfN(s)$.

$(ws)$: assume, for a contradiction, that for some $X,Y,Z\subseteq S$ it holds that $X \in qfN(s)$ but $X \cup Y \notin qfN(s)$ and $(S\backslash X) \cup Z \notin qfN(s)$. W.l.o.g. we assume that there are $t_1,u_1$ such that $sRt_1$, $sRu_1$ and $t_1\in X\cup Y$ but $u_1\notin X\cup Y$, and there are $t_2, u_2$ such that $sRt_2$, $sRu_2$ and $t_2 \notin (S\backslash X)\cup Z$ but $u_2\in  (S\backslash X)\cup Z$. Then $t_2\in X$ and $u_1\notin X$, which is contrary to the fact that $X\in qfN(s)$ and $sRu_1,sRt_2$.
\end{proof}

The following result is immediate by Props.~\ref{prop.rel-equiv-qf} and \ref{prop.qf-model}.

\begin{corollary}\label{coro.Krip-quasi}
For {\bf CL}, every Kripke model has a pointwise equivalent quasi-filter model.
\end{corollary}

However, for {\bf CL}, not every quasi-filter model has a pointwise equivalent Kripke model. The point is that quasi-filter models may not be closed under {\em infinite} (i.e. arbitrary) intersections (see the property $(r)$ in Def.~\ref{def.properties}).
\begin{proposition}\label{prop.quasi-filter-infinite}
For {\bf CL}, there is a quasi-filter model that has no pointwise equivalent Kripke model.
\end{proposition}

\begin{proof}
Consider an infinite model $\M=\lr{S,N,V}$, where
\begin{itemize}
\item $S=\mathbb{N}$,
\item for all $s\in S$, $N(s)=\{S,\emptyset,\{2n\text{ for some }n\in\mathbb{N}\},S\backslash\{2n\text{ for some }n\in\mathbb{N}\},\bigcap_{fin}S\backslash\{2n\text{ for some }n\in\mathbb{N}\},\bigcup_{fin}\{2n\text{ for some }n\in\mathbb{N}\}\}$,\footnote{$\bigcup_{fin}\{2n\text{ for some }n\in\mathbb{N}\}$ denotes the union of finitely many sets of the form $\{2n\text{ for some }n\in\mathbb{N}\}$, e.g. $\{0\}\cup\{2\}\cup\{4\}$.}
\weg{\item for all $s\in S$,
      \begin{itemize}
      \item for any even number $s$, $N(s)=\{S,\emptyset,\{2n\text{ for some }n\in\mathbb{N}\},S\backslash\{2n\text{ for some }n\in\mathbb{N}\},\bigcap_{fin}S\backslash\{2n\text{ for some }n\in\mathbb{N}\},\bigcup_{fin}\{2n\text{ for some }n\in\mathbb{N}\}\}$.\footnote{$\bigcup_{fin}\{2n\text{ for some }n\in\mathbb{N}\}$ means the union of finitely many sets of the form $\{2n\text{ for some }n\in\mathbb{N}\}$, e.g. $\{0\}\cup\{2\}\cup\{4\}$.}
      \item for any odd number $s$, $N(s)=\{S,\emptyset\}$.
      \end{itemize}}
\item $V(p)=\{2n\mid n\in\mathbb{N}\}$, $V(p_m)=\{m\}$ for all $m\in\mathbb{N}$.
\end{itemize}
It is not hard to check that $\M$ is a quasi-filter model.\footnote{To verify $(ws)$, we need only show the nontrivial case $\bigcup_{fin}\{2n\text{ for some }n\in\mathbb{N}\}$. For this, we show a stronger result: for all $Z\subseteq S$, $\bigcap_{fin}S\backslash\{2n\text{ for some }n\in\mathbb{N}\}\cup Z\in N(s)$. The cases for $Z=S$ or $Z=\emptyset$ are clear. For other cases, we partition the elements in $Z$ into three disjoint (possibly empty) parts: odd numbers, even numbers in $\bigcup_{fin}\{2n\text{ for some }n\in\mathbb{N}\}$, even numbers in $\bigcap_{fin}S\backslash\{2n\text{ for some }n\in\mathbb{N}\}$. Note that the first and third parts all belong to $\bigcap_{fin}S\backslash\{2n\text{ for some }n\in\mathbb{N}\}$; moreover, the union of the second part and $\bigcap_{fin}S\backslash\{2n\text{ for some }n\in\mathbb{N}\}$ is also in $N(s)$.} Note that for all $s\in S$, $p^\M\notin N(s)$, thus $\M,s\not\Vvdash \Delta p$. In particular, $\M,0\not\Vvdash\Delta p$.

Suppose that there is a pointwise equivalent Kripke model $\M'$, then $\M',0\nvDash\Delta p$. Thus there must be $2m$ and $2n+1$ that are accessible from $0$, where $m,n\in\mathbb{N}$. Since $p_{2m}^{\M'}=p_{2m}^{\M}=\{2m\}$, thus $\M',0\nvDash\Delta p_{2m}$.

However, since $p_{2m}^\M=\{2m\}\in N(0)$, we obtain $\M,0\Vvdash\Delta p_{2m}$, which is contrary to the supposition and $\M',0\nvDash\Delta p_{2m}$, as desired.
\end{proof}

However, when we restrict quasi-filter models to finite cases, the situation will be different.

\begin{proposition}\label{prop.quasi-Krip}
For every finite quasi-filter model $\M$, there exists a pointwise equivalent Kripke model $\M'$, that is, for all $\phi\in\mathcal{L}_\Delta$, for all worlds $s$,  $\M',s\vDash\phi\iff\M,s\Vvdash\phi$, i.e., $\phi^{\M'_\vDash}=\phi^\M$.
\end{proposition}

\begin{proof}
Let $\M=\lr{S,N,V}$ be a quasi-filter model. Define $\M'=\lr{S,R,V}$, where $R$ is defined as follows: for any $s,t\in S$,
$$sRt\iff t\in X\text{ for some }X\in N(s)\text{ and }\{t\}\notin N(s).$$
We will show that for all $\phi\in\mathcal{L}_\Delta$ and all $s\in S$, we have that $$\M',s\vDash\phi\iff \M,s\Vvdash\phi.$$

The proof proceeds with induction on $\phi\in\mathcal{L}_\Delta$. The nontrivial case is $\Delta\phi$, that is to show, $\M',s\vDash\Delta\phi\iff\M,s\Vvdash\Delta\phi$.

``$\Longleftarrow$:'' Suppose, for a contradiction, that $\M,s\Vvdash\Delta\phi$, but $\M',s\nvDash\Delta\phi$. Then $\phi^\M\in N(s)$, and there are $t,u\in S$ such that $sRt$ and $sRu$ and $\M',t\vDash\phi$ and $\M',u\nvDash\phi$. Since $\phi^\M\in N(s)$, by $(c)$, we get $S\backslash\phi^\M\in N(s)$; moreover, by $(ws)$, we obtain that $\phi^\M\cup\{u\}\in N(s)$ or $S\backslash\phi^\M\cup \{t\}\in N(s)$. If $\phi^\M\cup\{u\}\in N(s)$, then by $S\backslash \phi^\M\in N(s)$ and $(i)$, we derive that $(\phi^\M\cup\{u\})\cap S\backslash\phi^\M\in N(s)$, i.e., $\{u\}\cap S\backslash\phi^\M\in N(s)$, by induction hypothesis, $\{u\}=\{u\}\cap S\backslash\phi^{\M'_\vDash}\in N(s)$, contrary to $sRu$ and the definition of $R$. If $S\backslash\phi^\M\cup \{t\}\in N(s)$, similarly we can show that $\{t\}\in N(s)$, contrary to $sRt$ and the definition of $R$.

``$\Longrightarrow$:'' Suppose that $\M,s\not\Vvdash\Delta\phi$, to show that $\M',s\nvDash\Delta\phi$, that is, there are $t,u\in S$ such that $sRt,sRu$ and $\M',t\vDash\phi$ and $\M',u\vDash\neg\phi$. By supposition, $\phi^\M\notin N(s)$. By $(n)$ and $(c)$, $S\in N(s)$ and $\emptyset\in N(s)$.

Now consider the truth set of $\phi$ in $\M$, namely, $\phi^{\M_\vDash}=\{x\in S\mid \M,x\vDash\phi\}$. Clearly, $\phi^{\M_\vDash}\neq S$ and $\phi^{\M_\vDash}\neq \emptyset$. We show that there is a $t\in \phi^{\M_\vDash}$ such that $\{t\}\notin N(s)$ as follows: if not, i.e., for all $t\in \phi^{\M_\vDash}$ we have $\{t\}\in N(s)$, then by $(c)$, we get $S\backslash \{t\}\in N(s)$, and using $(i)$ we obtain $\bigcap_{t\in \phi^{\M_\vDash}}S\backslash \{t\}\in N(s)$, viz. $S\backslash \phi^{\M_\vDash}\in N(s)$.\footnote{Since $\M$ is finite, we need only use the property that $N$ is closed under finite intersections, which is equivalent to the property $(i)$. This is unlike the case in Prop.~\ref{prop.quasi-filter-infinite}.} Therefore using $(c)$ again, we conclude that $\phi^{\M_\vDash}\in N(s)$, which contradicts with the supposition and induction hypothesis.

Therefore, there is a $t\in \phi^{\M_\vDash}$ such that $\{t\}\notin N(s)$. Since $t\in S$ and $S\in N(s)$, by the definition of $R$, it follows that $sRt$; furthermore, from $t\in \phi^{\M_\vDash}$ and induction hypothesis, it follows that $\M',t\vDash\phi$.

Similarly, we can show that there is a $u\in (\neg\phi)^{\M'_\vDash}$ such that $\{u\}\notin N(s)$. Thus $sRu$ and $\M',u\vDash\neg\phi$, as desired.
\end{proof}

In spite of Prop.~\ref{prop.quasi-filter-infinite}, as we shall see in Coro.~\ref{coro.coincidelogicrel}, logical consequence relations over Kripke semantics and over the new neighborhood semantics on quasi-filter models coincide with each other for $\CL$.

\weg{\begin{corollary}\label{coro.coincidelogicrel}
The logical consequence relations $\Vvdash_q$ and $\vDash$ coincide for $\CL$. That is, for all $\Gamma\cup\{\phi\}\subseteq \mathcal{L}_\Delta$,
$\Gamma\Vvdash_q\phi\iff \Gamma\vDash \phi,$
where, by $\Gamma\Vvdash_q\phi$ we mean that, for every quasi-filter model $\M$ and $s$ in $\M$, if $\M,s\Vvdash \Gamma$, then $\M,s\Vvdash\phi$. Therefore, for all $\phi\in \mathcal{L}_\Delta$, $\Vvdash_q\phi\iff\vDash \phi$, i.e., the new semantics over quasi-filter models has the same logic (valid formulas) on {\bf CL} as the Kripke semantics.
\end{corollary}}

\weg{However, the converse does not hold: this is because quasi-filter model is not closed under supersets, thus is not augmented, but there is a one-to-one correspondence between augmented neighborhood models and Kripke models for $\CL$.

This will help us to find the axiomatization for $\CL$ on the class of $i$-frames under the original neighborhood semantics.}

\section{$qf$-Bisimulation}

This section proposes the notion of bisimulation for {\bf CL} over quasi-filter models, called `$qf$-bisimulation'. The intuitive idea of the notion is similar to monotonic $c$-bisimulation and $c$-bisimulation, i.e. the notion of precocongruences with particular properties (in the current setting, those four properties of quasi-filter models).

\begin{definition}[qf-bisimulation]
Let $\M=\lr{S,N,V}$ and $\M'=\lr{S',N',V'}$ be quasi-filter models. A nonempty relation $Z\subseteq S\times S'$ is a {\em qf-bisimulation} between $\M$ and $\M'$, if for all $(s,s')\in Z$,

\textbf{(qi)} $s\in V(p)$ iff $s'\in V'(p)$ for all $p\in\BP$;

\textbf{(qii)} if the pair $(U,U')$ is $Z$-coherent, then
$U\in N(s)\text{ iff }U'\in N'(s').$

We say $(\M,s)$ and $(\M',s')$ are {\em qf-bisimilar}, written $(\M,s)\bis_{qf}(\M',s')$, if there is a qf-bisimulation $Z$ between $\M$ and $\M'$ such that $(s,s')\in Z$.
\end{definition}

Note that the notion of $qf$-bisimulation is defined between quasi-filter models. It is clear that every qf-bisimulation is a $c$-bisimulation, but it is not necessarily a monotonic $c$-bisimulation, since it is easy to find a quasi-filter model which is not closed under supersets.\weg{since as illustrated by the following example, quasi-filter models may not be closed under supersets.

\begin{example}\label{example.not-superset}
The following $\mathcal{F}$ is a quasi-filter frame, but $\mathcal{F}$ is not closed under supersets, since for instance, $\emptyset\in N(s)$ but $\{s\}\notin N(s)$. (Here we use an arrow from a world $s$ to a set $X$ to mean that $X\in N(s)$.)
$$
\xymatrix{\{s,t\}&\emptyset\\
s\ar[u]\ar[ur]&t\ar[ul]\ar[u]\\}
$$
\end{example}}

Analogous to the case for $c$-bisimulation in Sec.~\ref{sec.c-bis}, we can show that
\begin{proposition}\label{prop.invariance-qf-bis}
Let $\M$, $\M'$ be both quasi-filter models, $s\in \M$, $s'\in\M'$. If $(\M,s)\bis_{qf}(\M',s')$, then for all $\phi\in\mathcal{L}_\Delta$,
$\M,s\Vvdash\phi\iff\M',s'\Vvdash\phi.$
\end{proposition}

\begin{theorem}[Hennessy-Milner Theorem for qf-bisimulation]
Let $\M$ and $\M'$ be $\Delta$-saturated quasi-filter models, and $s\in\M$, $s'\in\M'$. If for all $\phi\in\mathcal{L}_\Delta$, $\M,s\Vvdash \phi\iff\M',s'\Vvdash\phi$, then $(\M,s)\bis_{qf}(\M',s')$.
\end{theorem}

We conclude this section with a comparison between the notion of qf-bisimulation and that of rel-$\Delta$-bisimulation in~\cite[Def.~6]{Bakhtiarietal:2017}.

\begin{definition}[rel-$\Delta$-bisimulation]
Let $\M=\lr{S,R,V}$ and $\M'=\lr{S',R',V'}$ be Kripke models. A nonempty relation $Z\subseteq S\times S'$ is a {\em rel-$\Delta$-bisimulation} between $\M$ and $\M'$, if for all $(s,s')\in Z$,

\textbf{(Atoms)} $s\in V(p)$ iff $s'\in V'(p)$ for all $p\in\BP$;

\textbf{(Coherence)} if the pair $(U,U')$ is $Z$-coherent, then
$$(R(s)\subseteq U\text{ or }R(s)\subseteq S\backslash U)\text{ iff }(R'(s')\subseteq U'\text{ or }R'(s')\subseteq S'\backslash U').$$
We say $(\M,s)$ and $(\M',s')$ are {\em rel-$\Delta$-bisimilar}, written $(\M,s)\bis_{rel}(\M',s')$, if there is a rel-$\Delta$-bisimulation $Z$ between $\M$ and $\M'$ such that $(s,s')\in Z$.
\end{definition}

The result below asserts that every rel-$\Delta$-bisimulation between Kripke models can be transformed as a qf-bisimulation between quasi-filter models.

\begin{proposition}\label{prop.rel-qf}
Let $\M=\lr{S,R,V}$ and $\M'=\lr{S',R',V'}$ be Kripke models.
If $Z$ is a rel-$\Delta$-bisimulation between $\M$ and $\M'$, then $Z$ is a qf-bisimulation between $qf(\M)$ and $qf(\M')$.
\end{proposition}

\begin{proof}
Suppose $Z$ is a rel-$\Delta$-bisimulation between $\M$ and $\M'$. By Prop.~\ref{prop.qf-model}, $qf(\M)$~and $qf(\M')$ are both quasi-filter models.
It suffices to show that $Z$ satisfies the two conditions of a qf-bisimulation. For this, assume that $(s,s')\in Z$. \textbf{(qi)} is clear from \textbf{(Atoms)}.

For \textbf{(qii)}: let $(U,U')$ be $Z$-coherent. We have the following line of argumentation: $U\in qfN(s)$ iff (by definition of $qfN$) $(R(s)\subseteq U\text{ or }R(s)\subseteq S\backslash U)$ iff (by \textbf{Coherence}) $(R'(s')\subseteq U'\text{ or }R'(s')\subseteq S'\backslash U')$ iff  (by definition of $qfN'$) $U'\in qfN'(s')$.
\end{proof}

\weg{With the previous preparations in hand, we can now show van Benthem Characterization Theorem for qf-bisimulation.

\begin{theorem}\label{thm.van-benthem-qfbis}
Every modal formula is equivalent to an $\mathcal{L}_\Delta$-formula over the class of quasi-filter models iff it is invariant under qf-bisimulation.
\end{theorem}

\begin{proof}
By Prop.~\ref{prop.invariance-qf-bis}, we only need to show the `if' direction. For this, suppose that a modal formula $\phi$ is invariant under qf-bisimulation. If we show that $\phi$ is invariant under rel-$\Delta$-bisimulation, then according to van Benthem Characterization Theorem for rel-$\Delta$-bisimulation, $\phi$ is equivalent to an $\mathcal{L}_\Delta$-formula, then we are done.

For any models $\M$ and $\M'$ and $s\in \M$, $s'\in\M'$, assume that $(\M,s)\bis_{rel}(\M',s')$. By Prop.~\ref{prop.rel-qf}, $(qf(\M),s)\bis_{qf}(qf(\M'),s')$. By supposition, $qf(\M),s\Vvdash\phi$ iff $qf(\M'),s'\Vvdash\phi$. Then using Prop.~\ref{prop.rel-equiv-qf}, $\M,s\vDash\phi$ iff $\M',s'\vDash\phi$. We have thus shown that $\phi$ is invariant under rel-$\Delta$-bisimulation, as desired.
\end{proof}

Similarly, we can obtain

\begin{theorem}\label{thm.van-benthem-qfbis2}
Every first-order $\mathcal{L}_1$-formula is equivalent to an $\mathcal{L}_\Delta$-formula over the class of quasi-filter models iff it is invariant under qf-bisimulation.
\end{theorem}}

We do not know whether the converse of Prop.~\ref{prop.rel-qf} also holds in the current stage. Note that this is important, since if it holds, then we can see clearly the essence of rel-$\Delta$-bisimulation, i.e. precocongruences with those four quasi-filter properties. We leave it for future work.

\section{Frame definability}

Recall that under the old neighborhood semantics, all the ten neighborhood properties in Def.~\ref{def.properties} are undefinable in $\mathcal{L}_\Delta$. In contrast, under the new semantics, almost all these properties are definable in the same language. The following witnesses the properties and the corresponding formulas defining them. Recall that $(c)$ is the minimal condition of neighborhood frames.
$$
\begin{array}{llll}
(n)&\Delta\top&(i)& \Delta p\land \Delta q\to \Delta (p\land q)\\
(s)&\Delta (p\land q)\to \Delta p\land \Delta q&(c)&\Delta p\lra\Delta \neg p\\
(d)&\nabla p&(t)&\Delta p\to p\\
(b)&p\to\Delta\nabla p&(4)&\Delta p\to\Delta\Delta p\\
(5)&\nabla p\to\Delta\nabla p&&\\
\end{array}
$$

\weg{$(n)$:~~~$\Delta\top$

$(i)$:~~~$\Delta p\land \Delta q\to \Delta (p\land q)$

$(s)$:~~~$\Delta (p\land q)\to \Delta p\land \Delta q$

$(c)$:~~~$\Delta p\to\Delta \neg p$

$(d)$:~~~$\nabla p$

$(t)$:~~~$\Delta p\to p$

$(b)$:~~~$p\to\Delta\nabla p$

$(4)$:~~~$\Delta p\to\Delta\Delta p$

$(5)$:~~~$\nabla p\to\Delta\nabla p$}

\begin{proposition}\label{prop.fr-def}
The right-hand formulas define the corresponding left-hand properties.
\end{proposition}

\begin{proof}

By Prop.~\ref{prop.def-c}, $\Delta p\lra\Delta\neg p$~defines $(c)$. For other properties, we take $(d)$ and $(b)$ as examples, which resort to the property $(c)$. Given any $c$-frame $\mathcal{F}=\lr{S,N}$.

Suppose that~$\mathcal{F}$~has~$(d)$, to show that~$\mathcal{F}\Vvdash\nabla p$. Assume, for a contradiction that there is a valuation~$V$ on $\mathcal{F}$, and $s\in S$, such that~$\M,s\not\Vvdash\nabla p$, where~$\M=\lr{\mathcal{F},V}$. Then~$p^\M\in N(s)$. On the one hand, by supposition, $S\backslash p^\M\notin N(s)$; one the other hand, by~$(c)$, $S\backslash p^\M\in N(s)$, contradiction. Conversely, assume that~$\mathcal{F}$~does not have~$(d)$, to show that~$\mathcal{F}\not\Vvdash\nabla p$. By assumption, there is an~$X$~such that~$X\in N(s)$~and~$S\backslash X\in N(s)$. Define a valuation $V$ on~$\mathcal{F}$~such that~$V(p)=X$, and let~$\M=\lr{\mathcal{F},V}$. Thus~$p^\M\in N(s)$, i.e.,~$\M,s\Vvdash \Delta p$, and hence~$\M,s\not\Vvdash \nabla p$.

Suppose~$\mathcal{F}$~has~$(b)$, to show~$\mathcal{F}\Vvdash p\to\Delta \nabla p$. For this, given any $\M=\lr{S,N,V}$ and $s\in S$, assume that~$\M,s\Vvdash p$, then~$s\in p^\M$. By supposition, $\{u\in S\mid S\backslash p^\M\notin N(u)\}\in N(s)$. By~$(c)$, this is equivalent to that~$\{u\in S\mid p^\M\notin N(u)\}\in N(s)$, i.e.,~$\{u\in S\mid \M,u\Vvdash \nabla p\}\in N(s)$, viz.,~$(\nabla p)^\M\in N(s)$, thus~$\M,s\Vvdash \Delta \nabla p$. Conversely, suppose~$\mathcal{F}$~does not have~$(b)$, to show~$\mathcal{F}\not\Vvdash p\to\Delta \nabla p$. By supposition, there is an~$s\in S$~and~$X\subseteq S$, such that~$s\in X$~and~$\{u\in S\mid S\backslash X\notin N(u)\}\notin N(s)$. Define a valuation $V$ on~$\mathcal{F}$~such that~$V(p)=X$, and let~$\M=\lr{\mathcal{F},V}$. Then~$\M,s\Vvdash p$, and~$\{u\in S\mid S\backslash p^\M\notin N(u)\}\notin N(s)$. By~$(c)$ again, this means that~$\{u\in S\mid p^\M\notin N(u)\}\notin N(s)$, that is,~$\{u\in S\mid \M,u\Vvdash \nabla p\}\notin N(s)$, i.e.,~$(\nabla p)^\M\notin N(s)$, therefore~$\M,s\not\Vvdash \Delta \nabla p$.
\end{proof}

The following result will be used in the next section.
\begin{proposition}
$\Delta p\to\Delta(p\to q)\vee\Delta(\neg p\to r)$ defines the property $(ws)$, where $(ws)$ is as defined in Def.~\ref{def.ws}.
\end{proposition}

\begin{proof}
Let $\F=\lr{S,N}$ be a neighborhood frame.

First suppose $\F$ has $(ws)$, we need to show $\F\Vvdash\Delta p\to\Delta(p\to q)\vee\Delta(\neg p\to r)$. For this, assume for any model $\M$ based on $\F$ and $s\in S$ that $\M,s\Vvdash\Delta p$. Then $p^\M\in N(s)$. By supposition, $p^\M\cup r^\M\in N(s)$ or $(\neg p)^\M\cup q^\M\in N(s)$. The former implies $(\neg p\to r)^\M\in N(s)$, thus $\M,s\Vvdash\Delta(\neg p\to r)$; the latter implies $(p\to q)^\M\in N(s)$, thus $\M,s\Vvdash\Delta(p\to q)$. Either case implies $\M,s\Vvdash\Delta(p\to q)\vee\Delta(\neg p\to r)$, hence $\M,s\Vvdash\Delta p\to\Delta(p\to q)\vee\Delta(\neg p\to r)$. Therefore $\F\Vvdash\Delta p\to\Delta(p\to q)\vee\Delta(\neg p\to r)$.

Conversely, suppose $\F$ does not have $(ws)$, we need to show $\F\not\Vvdash\Delta p\to\Delta(p\to q)\vee\Delta(\neg p\to r)$. From the supposition, it follows that there are $X$, $Y$ and $Z$ such that $X\in N(s),X\subseteq Y$ and $Y\notin N(s)$, $S\backslash X\subseteq Z$ and $Z\notin N(s)$. Define $V$ as a valuation on $\F$ such that $V(p)=X$, $V(q)=Z$ and $V(r)=Y$. Since $p^\M=V(p)\in N(s)$, we have $\M,s\Vvdash\Delta p$. Since $X\subseteq Y$, $(\neg p\to r)^\M=X\cup Y=Y\notin N(s)$, thus $\M,s\not\Vvdash\Delta(\neg p\to r)$. Since $S\backslash X\subseteq Z$, $(p\to q)^\M=(S\backslash X)\cup Z= Z\notin N(s)$, and thus $\M,s\not\Vvdash\Delta(p\to q)$. Hence $\M,s\not\Vvdash\Delta p\to\Delta(p\to q)\vee\Delta(\neg p\to r)$, and therefore $\F\not\Vvdash\Delta p\to\Delta(p\to q)\vee\Delta(\neg p\to r)$.
\end{proof}

Note that in the above proposition, we do not use the property $(c)$, that is to say, it holds for the class of all neighborhood frames.

\section{Axiomatizations}

This section presents axiomatizations of $\mathcal{L}_\Delta$ over various classes of frames. The minimal system $\mathbb{E}^\Delta$ consists of the following axiom schemas and inference rule.


\[\begin{array}{ll}
\TAUT& \text{all instances of tautologies}\\


\EquiKw& \Delta\phi\lra \Delta\neg\phi\\

\REKw &\dfrac{\phi\lra\psi}{\Delta\phi\lra\Delta\psi}
\end{array}\]

Note that $\mathbb{E}^\Delta$ is the same as $\mathbb{CCL}$ in~\cite[Def.~7]{FanvD:neighborhood}. Recall that $(c)$ is the minimal neighborhood property. 

\begin{theorem}
$\mathbb{E}^\Delta$ is sound and strongly complete with respect to the class of $c$-frames.
\end{theorem}

\begin{proof}
Immediate by the soundness and strong completeness of $\mathbb{E}^\Delta$ w.r.t. the class of all neighborhood frames under $\Vdash$~\cite[Thm.~1]{FanvD:neighborhood} and Coro.~\ref{coro.c-valid-reduce}.
\weg{Soundness is immediate by Prop.~\ref{prop.def-c}.

For strong completeness, suppose for all $\Gamma\cup\{\phi\}\subseteq\mathcal{L}_\Delta$, we have that $\Gamma\Vvdash_c\phi$. By Coro.~\ref{coro.c-valid-reduce},
$\Gamma\Vdash\phi.$ Using~\cite[Thm.~2]{FanvD:neighborhood}, we obtain $\Gamma\vdash_{\mathbb{E}^\Delta}\phi$, as desired.}
\end{proof}


Now consider the following extensions of $\mathbb{E}^\Delta$, which are sound and strongly complete with respect to the corresponding frame classes. We omit the proof detail since it is straightforward.
\[ \begin{array}{|l|l|l|l|}
  \hline
  \text{notation}& \text{axioms}& \text{systems} & \text{frame classes} \\
  \hline
  \Delta\texttt{M} & \Delta(\phi\land\psi)\to\Delta\phi\land\Delta\psi& \mathbb{M}^\Delta=\mathbb{E}^\Delta+\Delta\texttt{M} & cs \\
  \Delta\texttt{C} & \Delta \phi\land\Delta\psi\to\Delta(\phi\land\psi)& \mathbb{R}^\Delta=\mathbb{M}^\Delta+\Delta\texttt{C} & csi\\
    \hline
\end{array}
\]

One may ask the following question: is $\mathbb{R}^\Delta+\Delta\top$ sound and strongly complete with respect to the class of filters, i.e. the frame classes possessing $(s),(i), (n)$? The answer is negative, since the soundness fails, although it is indeed sound and strongly complete with respect to the class of filters satisfying $(c)$.

Now consider the following axiomatization $\mathbb{K}^\Delta$, which is provably equivalent to $\mathbb{CL}$ in~\cite[Def.~4.1]{Fanetal:2015}.


\begin{definition}[Axiomatic system~$\mathbb{K}^\Delta$]\label{def.plkw} The axiomatic system~$\mathbb{K}^\Delta$~is the extension of $\mathbb{E}^\Delta$ plus the following axiom schemas:
\[\begin{array}{ll}
\KwTop&\Delta\top\\
\KwCon & \Delta\phi\land\Delta\psi\to\Delta (\phi\land\psi)\\
\KwDis & \Delta \phi\to \Delta (\phi\to \psi )\lor \Delta(\neg \phi\to \chi)\\
\end{array}\]
\end{definition}

\begin{theorem}\label{thm.comp-cl}
$\mathbb{K}^\Delta$ is sound and strongly complete with respect to the class of quasi-filter frames. 
\end{theorem}

\begin{proof}
Soundness is immediate by frame definability results of the four axioms.

For strong completeness, since every $\mathbb{K}^\Delta$-consistent set is satisfiable in a Kripke model (cf. e.g.~\cite{Fanetal:2015}), by Coro.~\ref{coro.Krip-quasi}, every $\mathbb{K}^\Delta$-consistent is satisfiable in a quasi-filter model, thus also satisfiable in a quasi-filter frame.
\end{proof}

Note that \weg{instead of using the existing completeness results, }the strong completeness of $\mathbb{E}^\Delta$ and of $\mathbb{K}^\Delta$ can be shown directly, by defining the canonical neighborhood function $N^c(s)=\{|\phi|\mid \Delta\phi\in s\}$.


\weg{\begin{corollary}\label{coro.coincidelogicrel}
The logical consequence relations $\Vvdash_q$ and $\vDash$ coincide for $\CL$. That is, for all $\Gamma\cup\{\phi\}\subseteq \CL$,
$$\Gamma\Vvdash_q\phi\iff \Gamma\vDash \phi.$$
\end{corollary}}

As claimed at the end of Sec.~\ref{Sec.qf-structures}, for {\bf CL}, although not every quasi-filter model has a pointwise equivalent Kripke model, logical consequence relations over Kripke semantics and over the new neighborhood semantics on quasi-filter models coincide with each other. Now we are ready to show this claim.

\begin{corollary}\label{coro.coincidelogicrel}
The logical consequence relations $\Vvdash_{qf}$ and $\vDash$ coincide for $\CL$. That is, for all $\Gamma\cup\{\phi\}\subseteq \mathcal{L}_\Delta$,
$\Gamma\Vvdash_{qf}\phi\iff \Gamma\vDash \phi,$
where, by $\Gamma\Vvdash_{qf}\phi$ we mean that, for every quasi-filter model $\M$ and $s$ in $\M$, if $\M,s\Vvdash \Gamma$, then $\M,s\Vvdash\phi$. Therefore, for all $\phi\in \mathcal{L}_\Delta$, $\Vvdash_{qf}\phi\iff\vDash \phi$, i.e., the new semantics over quasi-filter models has the same logic (valid formulas) on {\bf CL} as the Kripke semantics.
\end{corollary}

\begin{proof}
By the soundness and strong completeness of $\mathbb{K}^\Delta$ with respect to the class of all Kripke frames (cf. e.g.~\cite{Fanetal:2015}), $\Gamma\vdash_{\mathbb{K}^\Delta}\phi$ iff $\Gamma\vDash\phi$. Then using Thm.~\ref{thm.comp-cl}, we have that $\Gamma\Vdash_{qf}\phi$ iff $\Gamma\vDash\phi$.
\end{proof}

\section{Conclusion and Discussions}

In this paper, we proposed a new neighborhood semantics for contingency logic, which simplifies the original neighborhood semantics in~\cite{FanvD:neighborhood} but keeps the logic the same. This new perspective enables us to define the notions of bisimulation for contingency logic over various model classes, one of which can help us understand the essence of nbh-$\Delta$-bisimulation, and obtain the corresponding Hennessy-Milner Theorems, in a relatively easy way. Moreover, we showed that for this logic, almost all the ten neighborhood properties, which are undefinable under the old semantics, are definable under the new one. And we also had some simple results on axiomatizations. Besides, under the new semantics, contingency logic has the same expressive power as standard modal logic. We conjecture that our method may apply to other non-normal modal logics, such as logics of unknown truths and of false beliefs. We leave it for future work.

Another future work would be axiomatizations of monotonic contingency logic and regular contingency logic under the old neighborhood semantics. Note that our axiomatizations $\mathbb{M}^\Delta$ and $\mathbb{R}^\Delta$ are not able to be transformed into the corresponding axiomatizations under the old semantics, since our underlying frames are $c$-frames. For example, although we do have~$\Vvdash_{cs}\Delta(\phi\land\psi)\to\Delta\phi\land\Delta \psi$, we do {\em not} have~$\Vdash_s\Delta(\phi\land\psi)\to\Delta\phi\land\Delta \psi$; consequently, although $\mathbb{M}^\Delta$ is sound and strongly complete with respect to the class of $cs$-frames under the new neighborhood semantics, it is {\em not} sound with respect to the class of $s$-frames under the old one. Thus the axiomatizations of these logics under the old neighborhood semantics are still open.\footnote{Update: These two open questions have been answered in~\cite{Fan:2017b}.} 

\bibliographystyle{plain}
\bibliography{biblio2014,biblio2016}
\end{document}